\documentclass[letter,11pt]{article}
\usepackage{graphicx}
\usepackage{amsmath}
\usepackage{amsfonts}
\usepackage[compact]{titlesec}
\usepackage{amsthm}
\usepackage{amssymb}
\usepackage{multirow}
\usepackage{color}
\usepackage{csquotes}
\usepackage{mathtools}
\usepackage{geometry}
\usepackage{dirtytalk}
\usepackage{algorithm}
\usepackage{algorithmicx, algpseudocode}

\usepackage{verbatim}
\usepackage{subcaption}

\newtheorem{theorem}{Theorem}[section]

\newtheorem{lemma}[theorem]{Lemma}

\theoremstyle{definition}

\newtheorem{remark}[theorem]{Remark}
\newcommand{\bfs}[1]{{\boldsymbol #1}}
\newcommand{\Tau}{\mathcal{T}}

\newcommand{\R}{\mathbb{R}}
\newcommand{\V}{\mathcal{V}}
\newcommand{\x}{\textbf{x}}

\newcommand{\0}{\bfs{0}}

\makeatletter
\renewcommand*{\@seccntformat}[1]{\csname the#1\endcsname\hspace{0.3em}}

\renewcommand\section{\@startsection{section}{1}{0pt}%
                                          {8pt plus 4pt minus 4pt}%
                                          {.01pt}%
                                          {\bf}}

\renewcommand\subsection{\@startsection{subsection}{2}{0pt}%
                                          {-1.3ex plus -0.5ex minus -.9ex}%
                                          {-2pt}%
                                          {\bf} }

\renewcommand\subsubsection{\@startsection{subsubsection}{3}{0pt}%
                                          {-.5ex plus -.5ex minus -.2ex}%
                                          {-10pt}%
                                          {\bf\em}}
\makeatother

\titlespacing{\section}{0pt}{2ex}{1ex}
\titlespacing{\subsection}{0pt}{-1ex}{0ex}
\titlespacing{\subsubsection}{0pt}{0.5ex}{0ex}

\def\d{\text{d}}

\usepackage{booktabs,dcolumn,caption}

\newcolumntype{d}[1]{D{.}{.}{#1}} 
\begin{document}
\begin{center}
{\LARGE On variational iterative methods for semilinear problems
}\\[2em]
{Prosper Torsu\\ \em{Department of Mathematics, California State University, Bakersfield, CA 93311, USA}}
\end{center}
\begin{abstract}
This paper presents an iterative method suitable for inverting semilinear problems which are important kernels in many numerical applications. The primary idea is to employ a parametrization that is able to reduce semilinear problems into linear systems which are solvable using fast Poisson solvers. Theoretical justifications are provided and supported by several experiments. Numerical results show that the method is not only computationally less expensive, but also yields accurate approximations.
\end{abstract}

\section{Introduction}
The object of discussion in this article is an iterative procedure for inverting the semilinear  problem
\begin{equation} \label{eq:bvp}
-\nabla \cdot (a \nabla \psi) + r(\x, \psi) = f(\x),~~~~\x \in \Omega,~~~~~~~\psi(\x) = 0,~~~\x \in \partial \Omega,\\
\end{equation}
where $\Omega$ is a bounded region with boundary $\partial \Omega$. The functions $f$ and $g$ are known continuous data and $r:\Omega \times \R \to \R$ is continuously differentiable. In a variety of context, \eqref{eq:bvp} is of great importance to mathematicians, physicists, biologists, engineers and other practitioners for modeling natural and human-induced phenomena. Models arising in this form describe steady-state characterization of evolution processes and have been the subject of extensive research in recent years. Some of the most commonly encountered classical mathematical models are prototypes of the differential equation in \eqref{eq:bvp}. They include the well-known Helmholtz equation, Bratu, Poisson-Boltzmann, the stationary Klein-Gordon, Allen-Cahn equation and Schr\"{o}dinger equations. Precisely, applications arise in, but not limited to the study of electrostatic properties \cite{baker2001electrostatics} and the study of acoustic and electromagnetic scattering \cite{colton2012inverse, martin2006multiple, voon2004helmholtz}. In the sphere of fluid dynamics, some schemes (e.g. the Implicit Pressure, Explicit Saturation formulation of water-oil displacement \cite{hasle2007geometric}) are governed by variants of \eqref{eq:bvp}.
\par Indeed, simpler variations of \eqref{eq:bvp} are separable in some coordinate systems, which allow solutions to be written in closed form. However, the fundamental models governing most natural phenomena involving motion, diffusion, reaction, etc. occur nonlinearly. Despite their pivotal contribution to the evolution of mathematical theories and experimentation, analytical methods have served minimal roles as only a handful of realistic models have known solutions. For several decades now, scientific computation has  provided critical information about a wide class of problems that would otherwise be unavailable via analytical means. From a computational standpoint, solving \eqref{eq:bvp} requires the simultaneous application of a numerical scheme with some sort of iterative procedure to an associated discrete algebraic system. As it is known, classical central and upwind schemes for the model problem \eqref{eq:bvp} do not always produce satisfactory approximations. This is due to the inherent complex dynamics of $r$ leading to numerical challenges such as the appearance of singular effects \cite{amrein2015fully, roos2008robust}. Even though these numerical tools are very successful, resolving instabilities for stable and accurate approximations often require a lot of computational time and large amounts of memory. 
\par Ideally, analytical solutions to differential equations are desired. While this is possible when some components of the equation are trivialized, realistic models rarely admit closed form solutions. In recent years, advances in differential and integro-differential equations saw the light of novel analytical techniques \cite{adomian1990review, aryana2018series, he2006homotopy, wazwaz1998comparison, yildirim2010he}. The heart of these methods relies on perturbation arguments and the conjecture that solutions to nonlinear equations can be expressed as infinite series. This leads to a set of explicit recursive relations governing the modes of the series. Methods of this sort are robust and efficient for approximating solutions to IVPs, but have been modified to adapt to boundary and initial boundary-value problems. Certainly, the principal theme that has commanded most discussions about the aforementioned techniques is the rapid convergence of the series solutions. Despite this conception, important questions about the convergence of the method are still open because related studies are mainly problem-specific and do not provide a unifying ground to generalize results. 
\par Far from a computational perspective is the issue of existence, solvability, uniqueness, regularity and number of solutions to \eqref{eq:bvp}. Nonlinear differential equations, in general, do not yield easily to rigorous and consolidated mathematical analysis. Successful investigations into the existence and uniqueness of solutions have relied on rigid assumption on $a, r$ and $f$. In the notes \cite{naito1984note}, Naito investigated the existence of entire solutions by enforcing separability of $r$ in the spatial coordinates and the unknown. In particular, the form $r = a(\x)g(w)$ was assumed, where $a$ is locally H\"older on $\R^n$ and $f$ is a locally Lipschitz continuous function on $\R^+$ which is positive and nondecreasing for $w>0$. In the said article, the author established approximate range of the limit values of positive solutions. Naito's results were later improved by Usami \cite{usami1986bounded}, where the general form $r = g(\x,w,\nabla w)$ was considered. A more general case was studied by Boccardo and co. \cite{boccardo2010semilinear}. Their study proved the existence and regularity of \eqref{eq:bvp} with a non-constant conductivity and $r = f(\x)w^{-\gamma}$, where $\gamma >0$. The article concluded that a locally and strictly positive distributional solution to \eqref{eq:bvp} exists if $f \in L^1(\Omega)$. A related discussion that runs across the aforementioned studies (also see \cite{amrein2015fully,badiale2010semilinear, lions1982existence}) is the existence of multiple nontrivial solutions even under very restrictive assumptions. Sometimes, infinitely many solutions exist \cite{kusano1986entire,ni1982elliptic}. As much as extensive study of special classes of nonlinear differential equations exist today, there is no unified theory of nonlinear equations \cite{rosinger2015can}. 
\newpage
\par Inspired by the monographs \cite{aryana2018series, lu2007explicit}, this study presents an iterative procedure for approximating the solution, if it exists, to the variational equation associated with \eqref{eq:bvp}. The method is developed by employing a decomposition that is able to reduce the problem into a collection of sub-equations governed by the Laplacian. For compatible boundary conditions and domains, the method allows Fast Poisson Solvers to be used. The method extends to other boundary types trivially. Prescribing pure Dirichlet or Neumann boundary conditions is a matter of convenience and allows Fourier transforms to be used.
\section{Variational formulation}\label{sec:wf}
Consider a bounded Lipschitz domain $\Omega \subset \R^2$ with polygonal boundary $\partial \Omega$. Let $a:\Omega \to \R^+$ with $0 < a_{\min} \leq a(\x) \leq a_{\max}$ for all $\x\in \Omega$ and define the usual Hilbert space
\begin{equation}
H^1_0(\Omega) = \big\{ w: w\in L^2(\Omega), ~\nabla w\in L^2(\Omega) ~\text{ and }~ w|_{\partial \Omega} = 0\big\}.
\end{equation}
Under the assumption that a solution (not necessarily unique) exists, given $f \in L^2(\Omega)$, the variational formulation associated with \eqref{eq:bvp} reads:
\begin{equation}\label{eq:vf}
\begin{split}
&~\text{Find $u \in H^1_0(\Omega)$ satisfying}\\[1ex]
&\int_\Omega a\nabla u \cdot \nabla w \,\d A + \int_\Omega r(\cdot, u)w \,\d A = \int_\Omega fw \,\d A ,~~~\forall w \in H^1_0(\Omega).
\end{split}
\end{equation}
In the sequel, the explicit dependence of $r$ on $\x$ is ignored and we write $r(\x,\psi) = r(\psi)$ for brevity. An important step towards the construction of the finite element solution consists of approximating the continuous problem in \eqref{eq:vf} by seeking the solution on a finite dimensional subspace of $H^1_0(\Omega)$. Mainly, $H^1_0(\Omega)$ is replaced by the set of piecewise continuous polynomials with compact support. 
\subsection{Galerkin approximation}\label{sec:vf}
Consider a triangulation $\Tau_h$ of $\Omega$ such that  ${\overline \Omega}=\cup_{\tau \in {\Tau}_h} \tau$ with $h=\max_{\tau \in {\Tau}_h} \{h_\tau\}$ and $h_\tau=\sqrt{|\tau|}$. The finite element space associated with ${\Tau}_h$ is defined as
\begin{equation}
\V_h  = \big\{w_h \in C(\overline \Omega): w_h|_\tau \text{ is linear for all } \tau \in \mathcal{T}_h, \text{ and } w_{h}|_{\partial\Omega} = 0 \big\}.
\end{equation}
Then, the associated Galerkin finite element formulation is:
\begin{equation} \label{eq:fem}
\begin{split}
&~\text{Find $u_h \in \mathcal{V}_h$ satisfying}\\[1ex]
&\int_\Omega a\nabla u_h \cdot \nabla w_h \,\d A + \int_\Omega r(\cdot, u_h)w_h \,\d A = \int_\Omega fw_h \d A,~~\forall w_h \in \mathcal{V}_h.
\end{split}
\end{equation}
Let $Z$ be the set of degrees of freedom associated with $\mathcal{V}_h$ having $|Z| = \widetilde N$. Then,
\begin{equation}
\mathcal{V}_h = \text{span} \big\{ \phi_z \big\}_{z \in Z}
\end{equation}
where $\phi_z$ is the finite element basis function tagged with the degree of freedom $z$. By setting $u_h = \sum_{z \in Z} \alpha_z \phi_z$, the variational problem in \eqref{eq:fem} reduces to the following discrete algebraic equation governing $\bfs{\alpha} \in \mathbb{R}^{\widetilde N}$:
\begin{equation}\label{eq:nonlinsys}
\mathcal{R}(\bfs{\alpha}) = \mathcal{A}\bfs{\alpha} + \mathcal{B}(\bfs{\alpha}) + \bfs{\beta} = 0,
\end{equation}
where $\mathcal{A} \in M_{\widetilde N \times \widetilde N}(\R),~\bfs{\beta}, \in \mathbb{R}^{\widetilde N}$ and $\mathcal{B}(\bfs{\alpha})$ is a matrix function given by 
\begin{equation}
\mathcal{A}_{\zeta z} = \int_\Omega a\nabla \phi_z \cdot \nabla \phi_\zeta\, \d A,~~~~\beta_\zeta = \int_\Omega f\phi_{\zeta}\, \d A~~~\text{and}~~~\mathcal{B}_\zeta(\bfs{\alpha}) = \int_\Omega r(\sum \alpha_z \phi_z) \phi_\zeta \,\d A.
\end{equation}
 respectively. The Galerkin solution that approximates $\psi$ is written
\begin{equation}\label{eq:galerkinsol}
u_h = \sum_{z \in Z} \alpha_z \phi_z.
\end{equation}
As $\mathcal{R}$ is nonlinear, direct inversion of \eqref{eq:nonlinsys} requires the use of some iterative scheme. Standard practice involves the application of quadrature rules alongside a linearization procedure. With appropriate regularity conditions, Newton methods guarantee higher convergence and are generally preferred. However, the laborious and unavoidable calculation of the associated Jacobian matrix makes Newton methods unappealing. Beside fixed point iterations, analytical methods have also been utilized to obtain qualitative features of solutions to nonlinear problems.  Discussed in the next section is an iterative method which provides a general machinery for inverting \eqref{eq:vf}.

\section{An iterative approach}\label{eq:im}
We furnish the weak form in \eqref{eq:vf} with a parametrization which transforms the variational problem into a sequential inversion of linear sub-problems. Suppose that $\psi$ and $r$ admit the decompositions
\begin{equation}\label{eq:uapolyseries} 
\psi  =  \sum_{\zeta \leq M} \psi_\zeta~~~~\text{and}~~~~r =  \sum_{\zeta \leq M} \mathcal{P}_\zeta
\end{equation}
where $\psi_\zeta,~0 \leq \zeta \leq M$, are to be determined. Here, construction of the set $\{ \mathcal{P}_\zeta\}_{\zeta \leq M}$ is not well-defined, and thus, gives the reader the freedom to explore several possibilities. The construction adopted in this study engages sets of Adomian polynomials \cite{he2006homotopy, yildirim2010he} defined as
\begin{equation}\label{eq:hepolyn}
\mathcal{P}_\zeta = 
\dfrac{1}{(\zeta-1)!} \dfrac{\textrm{d}^{\zeta - 1}}{\textrm{d} \lambda^{\zeta - 1}} \Big[r\Big(\sum_{j \geq 0} \lambda^j 
\psi_j \Big) \Big]_{\lambda = 0}.
\end{equation}
This way, $\mathcal{P}_\zeta$ depends continuously on $\psi_j,~ 0\leq j \leq \zeta-1$ only. Upon substituting \eqref{eq:uapolyseries} into \eqref{eq:vf}, the variational equation can be written in the form 
\begin{equation}\label{eq:adomdecompo}
 \sum_{\zeta \leq M} \left(\int_\Omega a\nabla \psi_\zeta \cdot \nabla w \,\d A + \int_\Omega \mathcal{P}_\zeta w \,\d A \right) = \int_\Omega fw \, \d A.
\end{equation}
The ultimate goal is to invert a set of linear sub-problems. Thus, the iterative scheme is developed by constructing the sequence $\{ \psi_\zeta\}_{\zeta \leq M}$ where the terms in \eqref{eq:adomdecompo} are associated such that $\psi_\zeta$ solves
\begin{equation}\label{eq:adommatchingun} 
\begin{split}
&~\text{Find $\psi_\zeta \in H^1_0(\Omega)$ satisfying}\\[1ex]
&\int_\Omega a\nabla \psi_\zeta \cdot \nabla w \,\d A = (\delta_{\zeta0} - 1)\int_\Omega \mathcal{P}_\zeta w \,\d A + \delta_{\zeta0} \int_\Omega fw \,\d A,~~~\forall w \in H^1_0(\Omega).
\end{split}
\end{equation}
where $\delta$ is the Kronecker delta. By construction, the term $\mathcal{P}_\zeta$ in the equation for $\psi_\zeta$ is known. Moreover, $\psi_0$ absorbs the boundary conditions of the original problem. All higher order terms serve as corrections to the approximation $\psi = \psi_0$ and honor the same boundary structure but homogeneous conditions. The practical solution is given by the $(M+1)$-term partial sum
\begin{equation}\label{eq:partialsol}
\psi_M = \sum_{\zeta \leq M} \psi_\zeta.
\end{equation}
By requiring that $\mathcal{P}_\zeta \in L^2(\Omega)$ for all $0\leq \zeta \leq M$, Lax-Milgram supplies the existence and uniqueness of each mode, and hence the partial solution.
Because the decomposition is performed on the variational equation, the approach is applicable to other boundary types with equal success. Choosing pure Dirichlet or Neumann condition are means to facilitate the use of Fast Fourier Transforms.

\remark
It is noteworthy that, as it is, this formulation may yield the trivial partial solution. Observe that, for a homogeneous equation, the zeroth mode solves $\int_\Omega a\nabla \psi_\zeta \cdot \nabla w \,\d A = 0$ which yields the $\psi_0 = 0$ upon enforcing homogeneous boundary conditions. Largely, the zeroth mode inherits the most important qualitative features of the solution. As $\psi_0$ deviates significantly, convergence is affected adversely. This can be avoided by introducing an auxiliary function $g \in L^2(\Omega)$ into the differential equation. For instance, one may rewrite the original equation  $-\nabla \cdot (a \nabla \psi) + r(\x, \psi) = 0$  as $-\nabla \cdot (a \nabla \psi) + \widetilde r(\x, \psi) = g(\x)$, where $ \widetilde r(x, \psi) =  r(\x, \psi) + g(\x)$. This leads to a new set of polynomials $\{\widetilde{\mathcal{P}}_\zeta\}_{\zeta \leq M}$, and a nontrivial solution zeroth mode available by solving $\int_\Omega \nabla \psi_0 \cdot \nabla w \,\d A = \int_\Omega g w \,\d A$.

\par Contrary to fixed point methods, this approximation eliminates any form of linearization. The following results provide theoretical justifications which support stability of the modes and also establish convergence of \eqref{eq:partialsol}.
\begin{lemma} \label{lem:stability}
For all $\zeta \leq M$, the partial solution in \eqref{eq:partialsol} is stable provided $\mathcal{P}_\zeta \in L^2(\Omega)$ for all ~$0 \leq \zeta \leq M$. 
\end{lemma}
\begin{proof}
Stability of each mode is immediate since the weak forms in \eqref{eq:adommatchingun} possess unique solutions in $H^1_0(\Omega)$ with 
$$
\| \nabla {\psi}_\zeta\|_{2} \leq  \delta_{\zeta 0}C_{\zeta}\|f\|_{2} + (1 - \delta_{\zeta 0})\widetilde C_{\zeta}\|\mathcal{P}_{\zeta}\|_{2},~~~\zeta =0,1,2,\cdots,M,
$$
where $\| \cdot \|_2 = \left(\int_\Omega |\cdot|^2 \, \d A \right)^{1/2}$ and the constants $C_{\zeta},\widetilde C_{\zeta},~\zeta = 0,1,2\cdots$ are independent of $\psi_\zeta$. It follows that  
\begin{equation*}
\| \nabla \psi_M\|_{2} \leq  C_{\zeta}\|f\|_{2} + \sum_{\zeta \leq M} \widetilde C_{\zeta}\|\mathcal{P}_{\zeta}\|_{2}.
\end{equation*}
This completes the proof.
\end{proof}

\begin{lemma} \label{lem:lem2}
Let $\psi$ satisfy \eqref{eq:vf} and $\psi_M$ be the partial solution in \eqref{eq:partialsol}. If $\mathcal{P}_\zeta \in L^2(\Omega)$ for all $0 \leq \zeta \leq M$, then $\|\nabla(\psi - \psi_M)\|_{2} \to 0$  as $M \to \infty$.
\end{lemma}
\begin{proof}
By assumption, $\psi$ solves \eqref{eq:vf} and $\psi_M$ satisfies
\begin{equation*}
\int_{\Omega} a\nabla \psi_M \cdot \nabla w\, \d A + \sum_{\zeta \leq M} \int_{\Omega} \mathcal{P}_\zeta w \, \d A = \int_\Omega fw \,\d A.
\end{equation*}
It follows that the error $e_M = \psi - \psi_M$ satisfies
\begin{equation*}
\begin{split}
\int_{\Omega} a\nabla e_M \cdot \nabla w\, \d A &= \int_\Omega r(\psi) w \,\d A - \sum_{\zeta \leq M} \int_{\Omega} \mathcal{P}_\zeta w \, \d A = \sum_{\zeta > M} \int_{\Omega} \mathcal{P}_\zeta w \, \d A.
\end{split}
\end{equation*}
From Cauchy-Schwarz, Poincar\'e and triangle inequalities, it holds that 
\begin{equation*}
\begin{split}
\|\nabla e_M\|^2_{2} &\leq \Bigg(\sum_{\zeta > M}  \|\mathcal{P}_\zeta\|_{2} \Bigg) \|e_M\|_{2} \leq \Bigg(\sum_{\zeta > M} \|\mathcal{P}_\zeta\|_{2} \Bigg) \|\nabla e_M\|_{2}. 
\end{split}
\end{equation*}
Consequently, 
\begin{equation}\label{eq:errorbound}
\|\nabla e_M\|_{2} \leq \sum_{\zeta > M}  \|\mathcal{P}_\zeta\|_{2}. 
\end{equation}
The theorem follows as $\sum_{\zeta > M} \mathcal{P}_\zeta \to 0$ for sufficiently large $M$.
\end{proof}
\subsection{Comparison with Picard iteration}
\par The iterative method discussed above shares some similarities with Picard Iteration, where each iteration solves a linear problem. For the form in \eqref{eq:vf}, the fixed point iteration is:
\begin{algorithm}[h!]
\caption{Picard iteration}\label{picard} 
\begin{algorithmic}[1] 
\State Set an initial iterate $\psi^{(0)} \in H^1_0(\Omega)$.
\For{$M = 0,1, \cdots$}
\State Find $\psi^{(M+1)} \in H^1_0(\Omega)$ satisfying
\medskip
\State $\displaystyle \int_\Omega a\nabla \psi^{(M+1)} \cdot \nabla w \,\d A = \int_\Omega fw \,\d A - \int_\Omega r(\psi^{(M)})w \,\d A,~~~\forall w \in H^1_0(\Omega).$
\smallskip
\EndFor
\end{algorithmic}
\end{algorithm}\\
In fact, Lemma \ref{lemm:lemmaadpisame} shows that if $r$ is linear and the set $\{ \mathcal{P}_\zeta\}_{\zeta \geq 0}$ is constructed according to \eqref{eq:hepolyn}, the two schemes are identical.
\begin{lemma}\label{lemm:lemmaadpisame}
Let $\psi_M$ and $\psi^{(M)}$ be the $M^{\text{th}}$ iterate of the procedure above and Picard iteration respectively. Assume that $r$ is linear and $\{\mathcal{P}_\zeta\}_{\zeta \geq 0}$ is constructed according to \eqref{eq:hepolyn}. If $\psi_0 = \psi^{(0)}$, then $\psi_M = \psi^{(M)}$ holds for all $M\geq 1$. 
\end{lemma}
\begin{proof}
The argument is inductive. Since $r$ is linear, we must have $\mathcal{P}_\zeta = r(\psi_\zeta) = \psi_\zeta$. Thus,
\begin{equation}\label{eq:vfmeth}
\begin{split}
\int_{\Omega} a\nabla \psi_M \cdot \nabla w\, \d A & = \int_\Omega fw \,\d A - \sum_{\zeta \leq M} \int_{\Omega} \mathcal{P}_\zeta w \, \d A\\
&=  \int_\Omega fw \,\d A - \int_\Omega \psi_{M-1} w\,\d A.
\end{split}
\end{equation}
From algorithm \ref{picard}, $\psi^{(M)}$ also satisfies
\begin{equation}\label{eq:vfpicard}
\int_\Omega a\nabla \psi^{(M)} \cdot \nabla w \,\d A = \int_\Omega fw\, \d A - \int_\Omega \psi^{(M-1)} w \,\d A.\\
\end{equation}
Suppose for the sake of induction that $\psi^{(M-1)} = \psi_{M-1}$ for $M \geq 2$. It follows from \eqref{eq:vfpicard} and \eqref{eq:vfmeth} that
\begin{equation*}
\begin{split}
\int_{\Omega} a\nabla(\psi_M - \psi^{(M)}) \cdot \nabla w\, \d A & = \int_\Omega \left(\psi^{(M-1)} - \psi_{M-1}\right) w\,\d A.
\end{split}
\end{equation*}
The last equation is zero by assumption. As this holds for arbitrary $w\in H^1_0(\Omega)$, it must be that $\psi^{(M)} - \psi_M$ is constant in $H^1_0(\Omega)$. We then deduce that $\psi^{(M)} - \psi_M = 0$.
\end{proof}
\par Equation \eqref{eq:errorbound} shows that convergence of the partial solution is controlled by the rate of decay of  $\mathcal{P}_\zeta$. Certainly, this procedure is attractive because, similar to Picard iteration, calculation of the Jacobian associated with \eqref{eq:nonlinsys} is replaced by the repeated inversion of the operator $\nabla \cdot(a \nabla \cdot )$. Since the computational cost of inverting a single instance of this operator on a highly resolved mesh can be enormous for traditional numerical methods, calculation of \eqref{eq:partialsol} is highly prohibitive for large $M$. More importantly, engagement of Fast Poisson Solvers is limited because symmetry of the resulting stiffness matrix coefficients is affected. 
\par In what follows, we reduce \eqref{eq:adommatchingun} into variational sub-problems governed by $\widetilde a(\x) = 1$ for all $\x$. This allows us to employ discrete sine transforms to accelerate calculation of the modes. To accomplish this, consider the conductivity log characterization \cite{gelhar1993stochastic}
\begin{equation}\label{eq:cond}
\mathcal{Y} = \log a,
\end{equation}
which exists, and is defined since $a(\x) >0$ for all $\x \in \Omega$.  This suggests that for sufficiently large $N$, $a$ can be approximated using the truncated Taylor series 
\begin{equation}\label{eq:consum}
a = \sum_{\xi \leq N}\mathcal{Y}^\xi/\xi!.
\end{equation}
Using this representation for a fixed $\zeta$, the variational equation in \eqref{eq:adommatchingun} becomes
\begin{equation}\label{eq:adommatchingunsum} 
\sum_{\xi \leq N}\frac{1}{\xi!}\int_\Omega \mathcal{Y}^\xi\nabla \psi_{\zeta \xi} \cdot \nabla w \,\d A = (\delta_{\zeta0} - 1)\int_\Omega \mathcal{P}_\zeta w \,\d A + \delta_{\zeta0} \int_\Omega fw \,\d A.
\end{equation}
By collecting terms of the same order (index of $\psi_\zeta$ and the exponent of $\mathcal{Y}$ adds up to the same $n \in \mathbb{N}$), \eqref{eq:adommatchingunsum}  transforms into the following recursive equations:
\begin{equation}\label{eq:adommatchingunsum} 
\begin{split}
&\int_\Omega \nabla \psi_{\zeta 0} \cdot \nabla w \,\d A = (\delta_{\zeta0} - 1)\int_\Omega \mathcal{P}_\zeta w \,\d A + \delta_{\zeta0} \int_\Omega fw \,\d A, \\[2ex]
&\int_\Omega \nabla \psi_{\zeta \xi} \cdot \nabla w \,\d A  + \sum_{j=1}^\xi \frac{1}{j!} \int_\Omega \mathcal{Y}^j\nabla \psi_{\zeta ,\xi-j} \cdot \nabla w \,\d A = 0,~~\xi \geq 1.
\end{split}
\end{equation}
 The approximate solution is written as 
\begin{equation}\label{eq:doublepartialsol}
\psi_{MN} = \sum_{\zeta \leq M}\sum_{\xi \leq N} \psi_{\zeta \xi}.
\end{equation}
Clearly, any unfavorable effect of $a$ on the invertibility of the resulting stiffness matrix is avoided. A minimal penalty is realized in the calculation of the load vector.
\subsection{Discrete algorithm}
The procedure discussed in the preceding sections can be summarized as follows:

\begin{algorithm}
\caption{Series approximation}\label{series} 
\begin{algorithmic}[1] 
\State Fix $M,N$.
\For{$\zeta=0$ to $N$}
\For{$\xi=0$ to $M$}
\State Find $\displaystyle \psi_{\zeta0} \in \V_h:\int_\Omega \nabla \psi_{\zeta 0} \cdot \nabla w \,\d A = \delta_{\zeta0} \int_\Omega fw \,\d A - \delta_{\xi0} \int_\Omega \mathcal{P}_{\zeta}w \,\d A$
\medskip
\State Find $\displaystyle \psi_{\zeta \xi} \in \V_h:\int_\Omega \nabla \psi_{\zeta \xi} \cdot \nabla w \,\d A + \sum_{j \leq\xi} \frac{1}{j!} \int_\Omega \mathcal{Y}^j\nabla \psi_{\zeta ,\xi-j} \cdot \nabla w \,\d A = 0,~~\zeta,\xi \geq 1$
\EndFor
\EndFor
\State Set $\psi_{MN} = \sum_{\zeta \leq M}\sum_{\xi \leq N} \psi_{\zeta \xi}$.
\end{algorithmic}
\end{algorithm}
For all $\zeta$, we can verify that the variational problems in Algorithm \ref{series} have unique solutions in $H^1_0(\Omega)$ satisfying the following stability estimate: 
\begin{equation}
\|\nabla \psi_{\zeta 0}\|_{2} \leq C_{\zeta0} \left(\delta_{\zeta0} \|f\|_{2} + \|\mathcal{P}_\zeta\|_{2}\right)
\end{equation}
provided $\mathcal{P}_\zeta \in L^2(\Omega)$. By construction, $\int_\Omega \nabla \psi_{\zeta1} \cdot \nabla w \,\d A + \int_\Omega \mathcal{Y} \nabla \psi_{\zeta0} \cdot \nabla w \,\d A = 0$. Thus, as long as $\mathcal{Y} \in L^2(\Omega)$, it follows that  
\begin{equation*}
\begin{split}
\|\nabla\psi_{\zeta 1}\|_{2} &\leq \|\mathcal{Y}\|_{2} \|\nabla \psi_{\zeta 0}\|_{2}\\
& \leq C_{\zeta1} \|\mathcal{Y}\|_{2} \left( \delta_{\zeta0} \|f\|_{2} + \|\mathcal{P}_\zeta\|_{2}\right).
\end{split}
\end{equation*}
By complete induction, we have
\begin{equation}
\|\nabla \psi_{\zeta \xi}\|_{2} \leq C_{\zeta \xi} \|\mathcal{Y}\|^{\xi}_{2}\left( \delta_{\zeta0} \|f\|_{2} + \|\mathcal{P}_\zeta\|_{2}\right).\\
\end{equation}
Upon a successful Galerkin discretization, the discrete equations have the form $\mathcal{M} \bfs{z} = \mathcal{F}$ where $\mathcal{M}$ is the discrete Laplacian. The collection can now be solved efficiently and fast using the fast Poisson solver discussed in the next section. As much as the summation in the equation for $\psi_{\zeta \xi}$ tells a strong dependence of higher-order terms on lower-order ones, Fast Poisson solvers comes in handy especially when a lot of terms are required to represent each mode accurately.

\par For Neumann problems, the Galerkin discretization is known to produce a linear system characterized by a one-dimensional kernel. The associated stiffness matrix is singular with fatal effects on iterative solvers, e.g., conjugate gradient. Often, the kernel is eliminated by specifying data at a particular node, or approximate the system using a minimization-based iterative solver \cite{bochev2005finite}. While these techniques seemingly avert direct inversion a singular matrix, Bochev and co., in the referenced article, remarked that such choices are subtle and unsatisfactory because specifying data at a specific node is ambivalent and roundoff errors may hinder convergence of the iterative solver. In this context, Algorithm \ref{series} provides the environment to exploit direct solvers. 
\par 
In very limited cases, \eqref{eq:bvp} is fast solvable on simple geometries. The most common numerical treatment of the finite element discretization involves a uniform mesh refinement across the entire computational domain. This poses numerical challenges as the finite element method (as well as finite difference and finite volume) performs poorly with vanishing element width. For the prescribed boundary conditions, this is no longer a limitation because we are able to use Fourier transforms to invert the sub-systems efficiently.
\section{A Fast Poisson Solver (FPS)}\label{sec:fps}
Fast Poisson Solvers are simple but powerful machinery which harness Fourier transforms to produce efficient solvers for the Poisson equation with compatible boundary conditions. They require a low amount of arithmetic operations and can be executed effectively with optimal operation count of $\mathcal{O} (n^2 \log n)$. The procedure that utilizes finite difference is discussed in \cite{iserles2009first}. Let $\mathcal{M} \bfs{z} = \mathcal{F}$ be the algebraic system resulting from the Galerkin approximation of
$$
\int_\Omega \nabla z \cdot \nabla w \,\d A = \int_\Omega gw \,\d A
$$
using the bilinear basis. Then $\mathcal{M}$ consists of block diagonal matrices identified as TST (Toeplitz, Symmetric and Tridiagonal). Denote by TD$(n; \alpha^\star, \beta^\star)$ the $n \times n$ tridiagonal matrix with $\alpha^\star$ along the principal diagonal and $\beta^\star$ along the sub- and super-diagonals. For the finite element discretization, let $\Delta x$ and $\Delta y$ be the element width with $n_1$ and $n_2$ elements in the vertical and lateral directions respectively. Then $\mathcal{M} =$ TD$(n; \mathcal{S}, \mathcal{T})$, where $\mathcal{S}$ and $\mathcal{T}$ are themselves TST with sub- and super-diagonals
\begin{equation}
\begin{split}
\mathcal{S}_{ii} &= \frac{4}{3\Delta x \Delta y}\,(\Delta x^2 + \Delta y^2), \hspace{0.23in} \mathcal{S}_{i-1,i} =  \mathcal{S}_{i,i+1} = \frac{1}{3\Delta x \Delta y}(\Delta x^2 - 2\Delta y^2),  \\[2ex]
\mathcal{T}_{ii} &= \frac{1}{3\Delta x \Delta y} (\Delta x^2 - 2\Delta y^2), \hspace{0.2in} \mathcal{T}_{i-1,i}  = \mathcal{T}_{i,i+1} = -\frac{1}{6\Delta x \Delta y}(\Delta x^2 + \Delta y^2). 
\end{split}
\end{equation}
Interestingly, TST matrices encode critical information which are fundamental to the implementation of FFT-based Poisson solvers. In fact, the linear system $\mathcal{M} \bfs{z} = \mathcal{F}$ is solvable completely by using the eigenvalues and eigenfunctions of $\mathcal{S}$ and $\mathcal{T}$ only. First, the eigenvectors are given explicitly by the formula 
\begin{equation} \label{eq:eigfun}
\varphi_{jl} = \sqrt{\frac{2}{n_{1}+1}}\sin \Big(\frac{\pi jl}{n_{1}+1} \Big), ~~~j,l=1,2,\cdots, n_{1}. 
\end{equation}
Moreover, the collection $\{\varphi_{jl} \}_{jl = 1}^{n_1}$ is orthonormal with corresponding eigenvalues
\begin{equation} \label{eq:eigvals}
\begin{split}
\lambda_j^\mathcal{S} &= \frac{4}{3}\left(\frac{\Delta x}{\Delta y} + \frac{\Delta y}{\Delta x}\right) + \frac{2}{3}\left(\frac{\Delta x}{\Delta y} - \frac{2\Delta y}{\Delta x}\right) \cos \Big(\frac{\pi j}{n_{1}+1} \Big),\\[3ex]
\lambda_j^\mathcal{T} &= \frac{1}{3}\left(\frac{\Delta y}{\Delta x} - \frac{2\Delta x}{\Delta y}\right) - \frac{1}{3}\left(\frac{\Delta x}{\Delta y} + \frac{\Delta y}{\Delta x}\right)\cos \Big(\frac{\pi j}{n_{1}+1} \Big).
 \end{split}
\end{equation}
With the set $\{\varphi_{jl},\lambda_j^\mathcal{S} , \lambda_j^\mathcal{T}\}_{jl = 1}^{n_1}$ available, construction of the fast Poisson solver is straightforward. Let $\boldsymbol{z} = (\boldsymbol{z}_1, \boldsymbol{z}_2, \cdots, \boldsymbol{z}_{n_{2}})^T$ and 
$\mathcal{F} = (\mathcal{Q}_1, \mathcal{Q}_2, \cdots, \mathcal{Q}_{n_{2}})^T$ where $\boldsymbol{z}_l$ and $\mathcal{Q}_l$ are block vectors corresponding to the $l^{\text{th}}$ column of $\mathcal{M}$ and $\mathcal{F}$ respectively. Then $\mathcal{M} \bfs{z} = \mathcal{F}$ is expressible as
\begin{equation}\label{eq:tridiag}
\mathcal{T}\boldsymbol{z}_{l-1} + \mathcal{S}\boldsymbol{z}_l + \mathcal{T}\boldsymbol{z}_{l+1} = \mathcal{Q}_l, ~~l=1,2,\cdots, n_{2}.
\end{equation}
Denote by $\Phi$ the matrix whose $(jl)^{\text{th}}$ entry is given by \eqref{eq:eigfun}. Then 
$$
\mathcal{S} = \Phi D_\mathcal{S} \Phi^T~~~\text{and}~~~\mathcal{T} = \Phi D_\mathcal{T} \Phi^T,
$$
where $D_\mathcal{I} = \text{diag}[\lambda_l^{\mathcal{I}},\lambda_l^{\mathcal{I}},\cdots,\lambda_l^{\mathcal{I}}]$. Orthogonality of $\Phi$ implies that \eqref{eq:tridiag} can be reduced to
\begin{equation}\label{eq:indetri}
D_\mathcal{T} \boldsymbol{\zeta}_{l-1} + D_\mathcal{S} \boldsymbol{\zeta}_l + D_\mathcal{T} \boldsymbol{\zeta}_{l+1} = \boldsymbol{\xi}_l, ~~~l=1,2,\cdots, n_{2},
\end{equation}
where, $ \Phi\boldsymbol{z}_l = \boldsymbol{\zeta}_l$ and  $ \Phi\mathcal{Q}_l = \boldsymbol{\xi}_l$. These operations successfully decoupled $\mathcal{M} \bfs{z} = \mathcal{F}$ into a set of independent one-dimensional tridiagonal systems which can be solved using fast tridiagonal solvers.

\section{Numerical Tests}
We present three numerical tests to explore the effect of $M$ and $N$ on the accuracy of the series approximations. Even though the problems are manufactured and may not bear any physical meaning, they do provide good insights into the performance of the method. Because of the close similarity, we also compute reference approximations by direct simulation of $\mathcal{R}(\bfs{\alpha}) = \0$ using Picard iteration. The stopping criterion is $\|\mathcal{R}(\bfs{\alpha})\|_\infty \leq 10^{-10}$.

\newpage
\subsection{Test 1}: In this foremost experiment, we set $a(x_1,x_2) = 1,~r(\psi) = \psi - \psi^3$ and $f$ is chosen appropriately so that \eqref{eq:bvp} admits  the solution $\psi(x_1,x_2) = \omega(x_1)\omega(x_2)$, where 
\begin{equation*}
\omega(s) = 1 + \exp(-1/0.05) - \exp(-s/0.05) - \exp((s - 1)/0.05). 
\end{equation*}
Since $a$ is constant, the index $\xi$ is redundant. Thus, we simply denote $\psi_{\zeta \xi}$ by $\psi_{\zeta}$. Figure \ref{fig:fig1} compares the numerical approximations of $\psi_{1}, \psi_{4}$ and $\psi_{8}$ with the analytical solution.
\begin{figure}[h!]
\centering
  \begin{subfigure}{5.5cm}
   \hspace*{0.02in}\makebox[0pt][l]{\includegraphics[height=1.2in, width=1.8in]{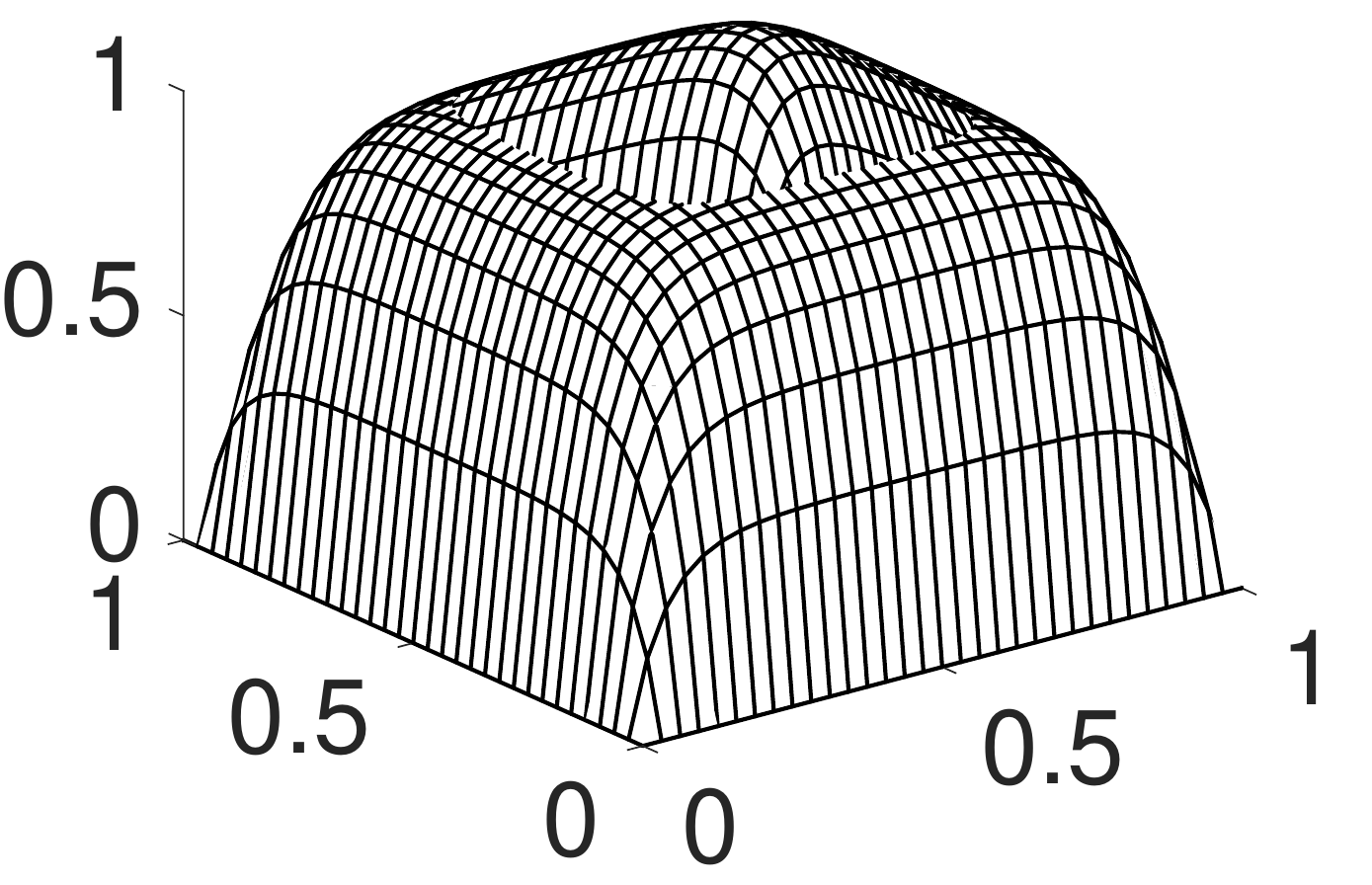}}
    \caption{$\psi_1$ approximation}
  \end{subfigure}
  \begin{subfigure}{5.2cm}
    \hspace*{0.06in}\makebox[0pt][l]{\includegraphics[height=1.2in, width=1.8in]{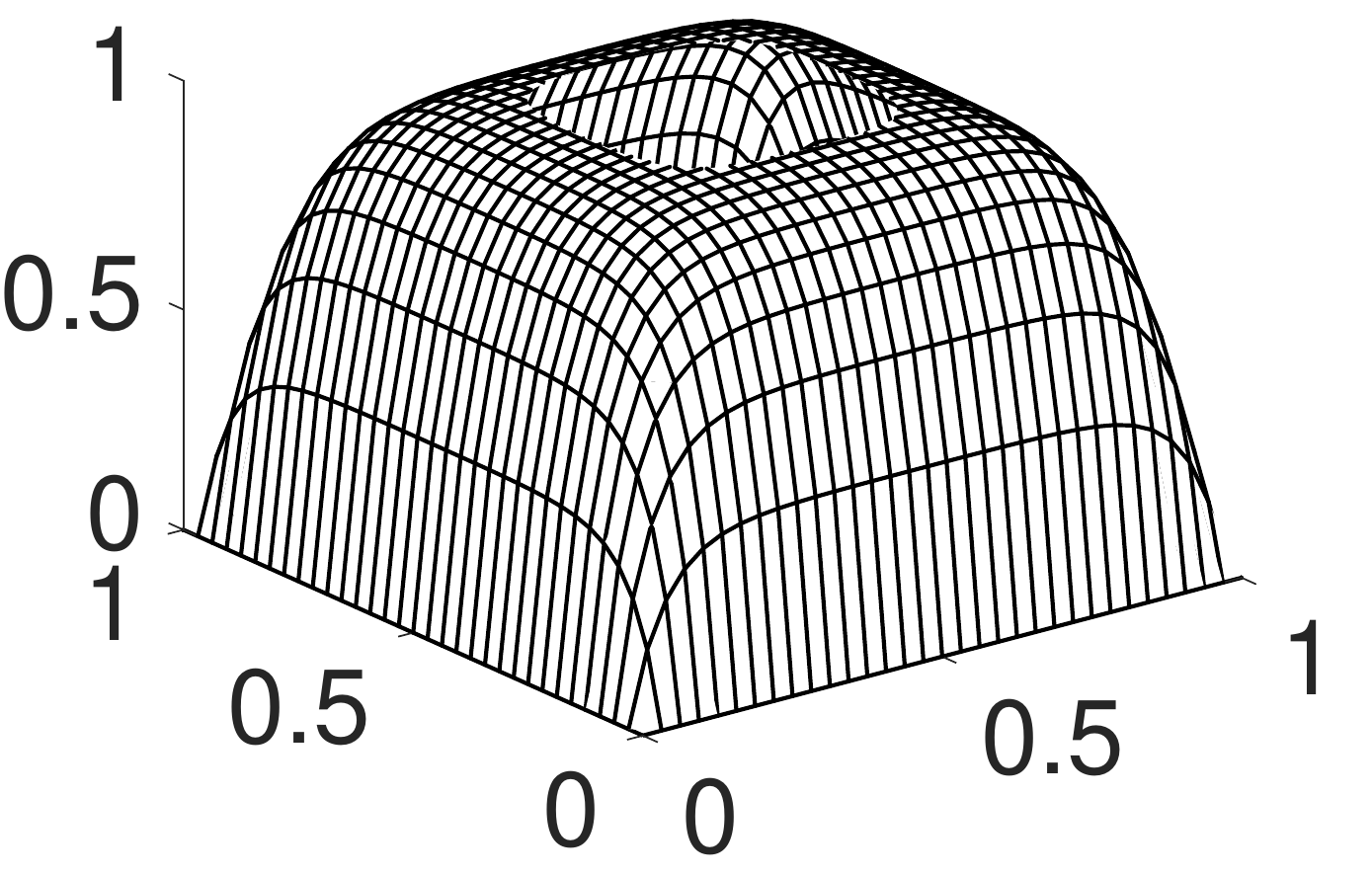}}
    \caption{$\psi_4$ approximation}
  \end{subfigure}\\[2ex]
\vspace*{0.1in}
  \begin{subfigure}{5.6cm}
    \makebox[0pt][l]{\includegraphics[height=1.2in, width=1.8in]{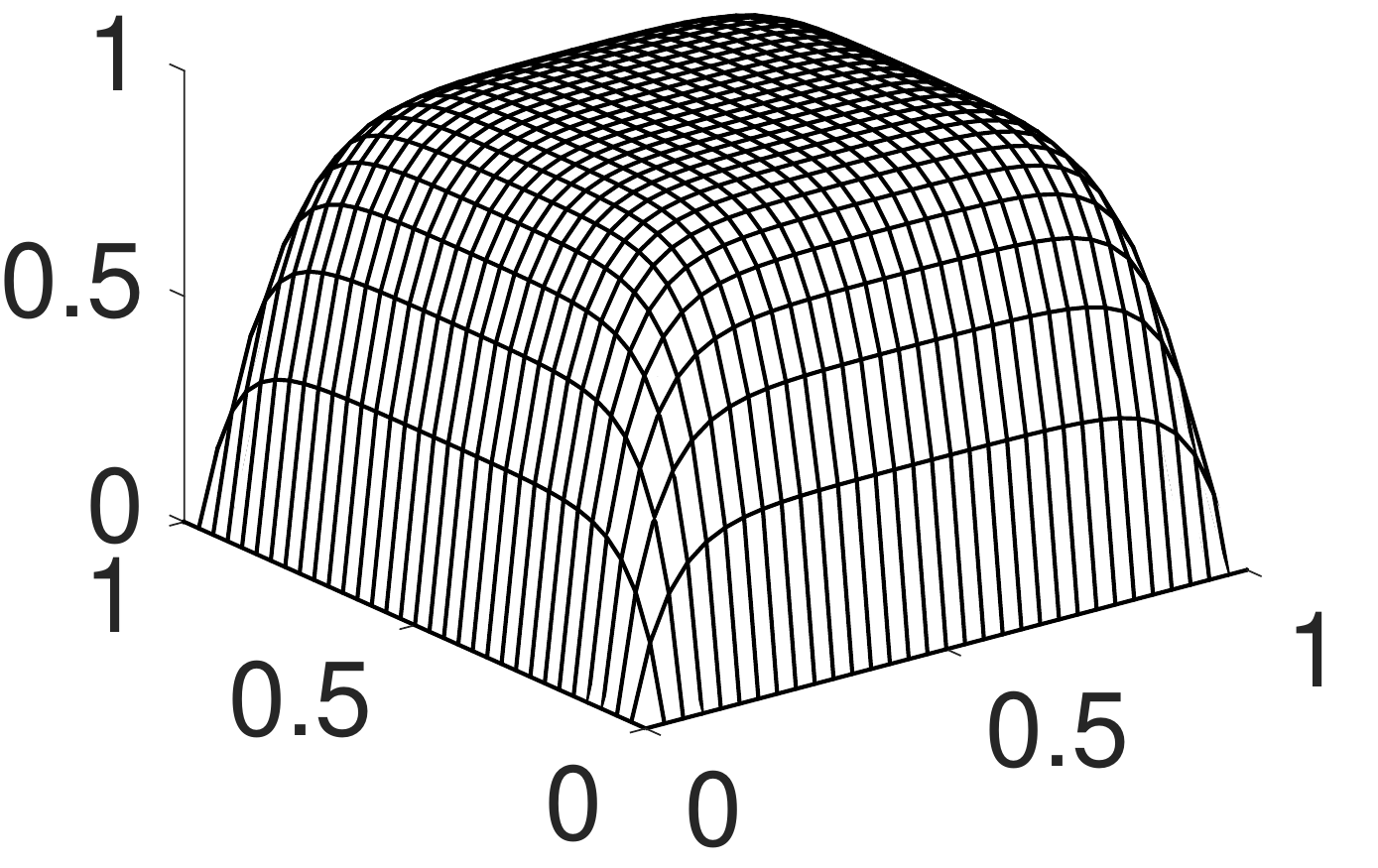}}
    \caption{$\psi_8$ approximation}
  \end{subfigure}
  \begin{subfigure}{5.5cm}
       \hspace*{0.08in}\makebox[0pt][l]{\includegraphics[height=1.2in, width=1.8in]{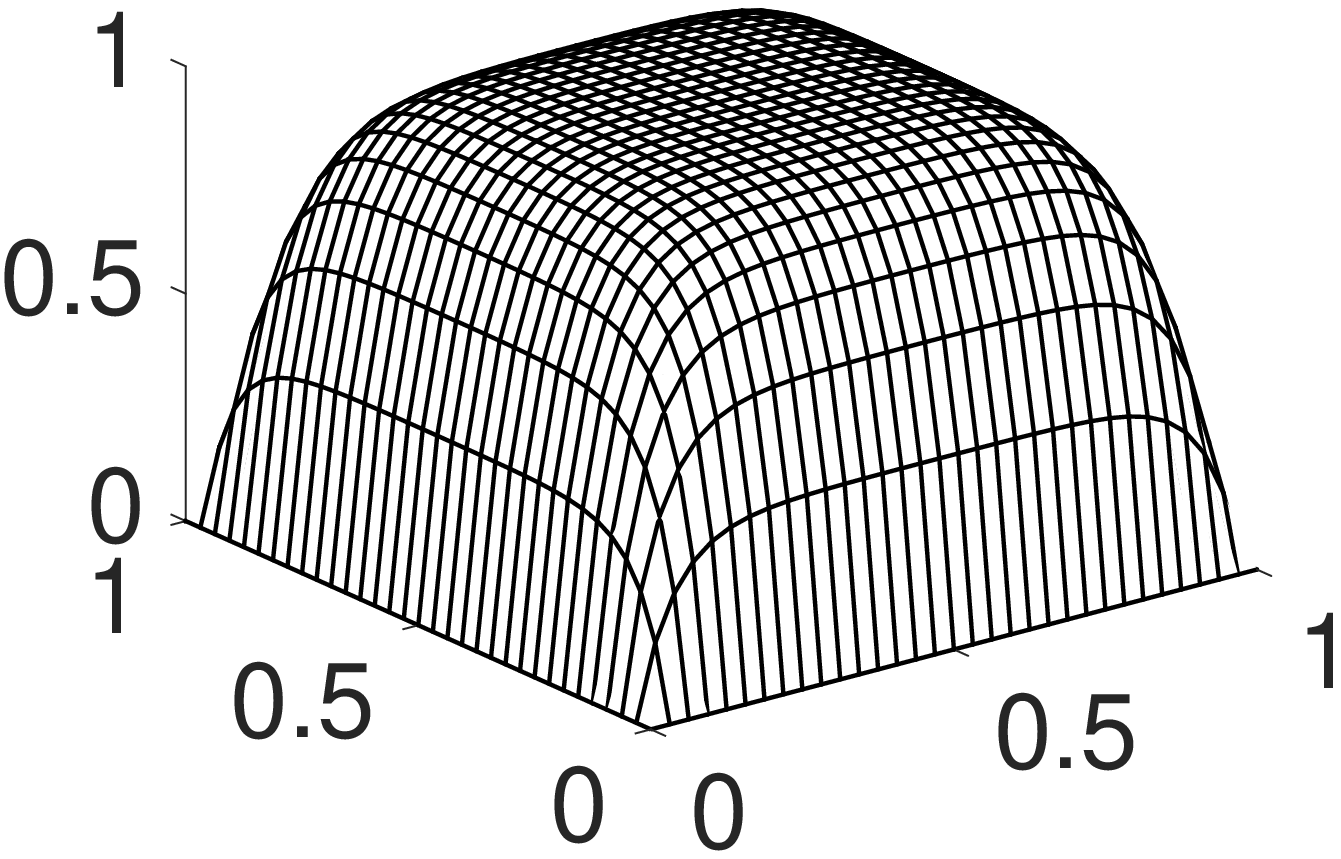}}
    \caption{Analytical solution}
  \end{subfigure}
 \caption {\small Approximation of various modes compared with the analytical solution.}
 \label{fig:fig1}
\end{figure}

\begin{table}[h!]
\setlength\tabcolsep{4.5pt} 
\begin{tabular*}{\textwidth}{@{\extracolsep{\fill}}l*{5}{d{2.9}} }
\toprule
 &  & \multicolumn{2}{c}{Number of elements}  & & \\
\midrule
 & \multicolumn{1}{c}{$~~2^{3}\times2^{3}$} & \multicolumn{1}{c}{$~~2^{4}\times2^{4}~$} & \multicolumn{1}{c}{$~~2^{5}\times2^{5} ~$} & \multicolumn{1}{c}{$~~2^{6}\times2^{6} ~$} & \multicolumn{1}{c}{$~~2^{7}\times2^{7}$} \\
\midrule
$M = 1$ & 3.80095(-01) & 9.56161(-02) & 2.99741(-02) & 1.41233(-02) & 1.02027(-02) \\[1ex]          
$M = 3$ & 3.88708(-01) & 9.61027(-02) & 2.98165(-02) & 1.38540(-02) & 9.90701(-03) \\[1ex]
$M = 5$ & 3.89504(-01) & 9.61117(-02) & 2.98152(-02) & 1.38520(-02) & 9.90491(-03)\\[1ex]
$M = 7$ & 3.89612(-01) & 9.61119(-02) & 2.98152(-02) & 1.38520(-02) & 9.90491(-03)\\
\midrule
Picard & 3.89633(-01) & 9.61119(-02) & 2.98152(-02) & 1.38520(-02) & 9.90491(-03) \\
\bottomrule
\end{tabular*}
\caption{\small{Relative error of series approximation for various $M$ and number of elements.}} 
\label{tab:tab1}
\end{table}

\newpage
\subsection{Test 2}: Here, we set $a(x_1,x_2) = \cosh(x_1+x_2),~r(\psi) = \sin \psi$ and $f$ is constructed such that \eqref{eq:bvp} admits a closed form solution 
$$
\psi(x_1,x_2) = 10(x_1-x_1^2)(x_2-x_2^2)\tan^{-1}[100(x_1-x_2)^6].
$$ 
\vspace*{-0.25in}
\begin{figure}[h!]
\centering
  \begin{subfigure}{5.5cm}
   \hspace*{0.02in}\makebox[0pt][l]{\includegraphics[height=1.2in, width=1.85in]{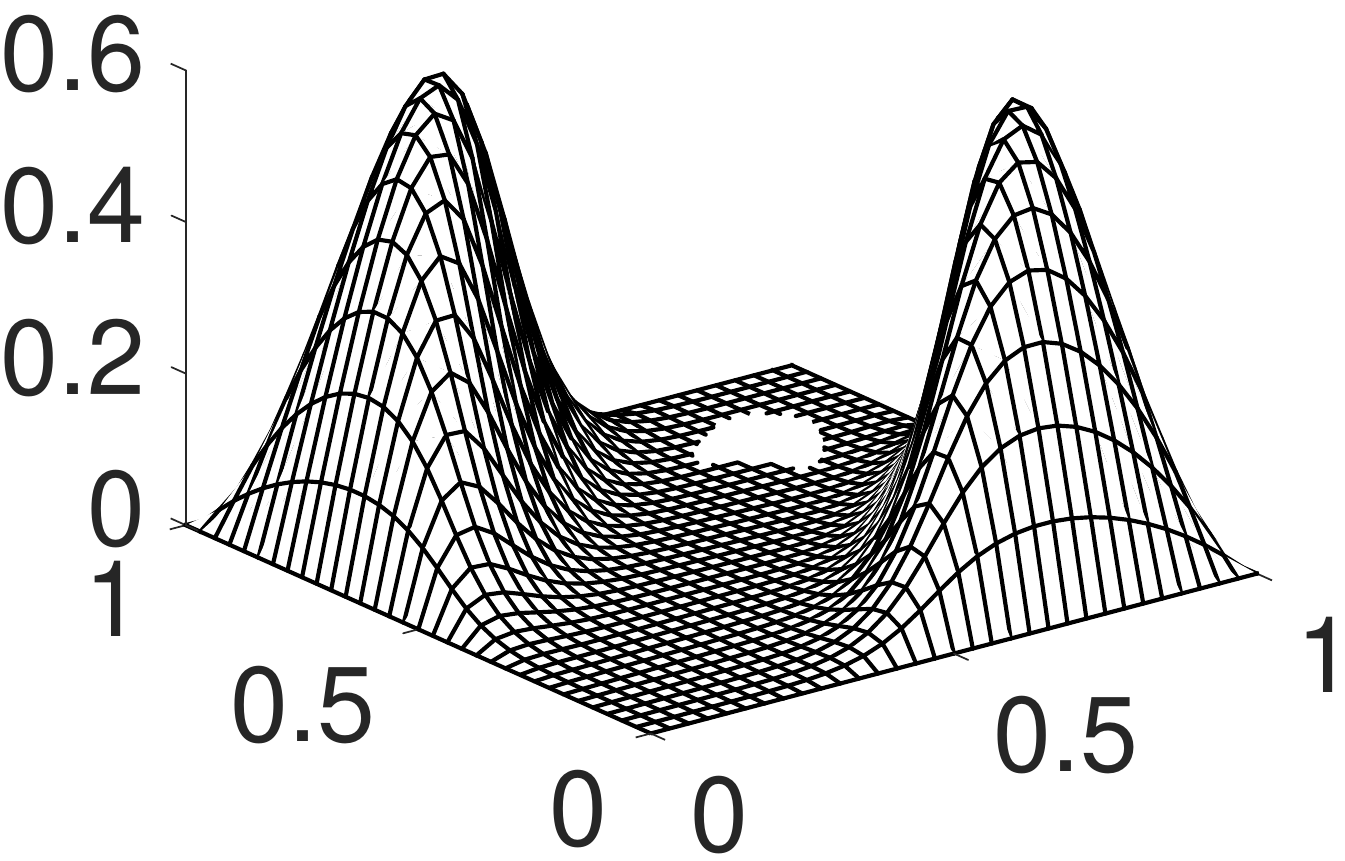}}
    \caption{$\psi_{10}$ approximation}
  \end{subfigure}
  \begin{subfigure}{5.2cm}
    \hspace*{0.06in}\makebox[0pt][l]{\includegraphics[height=1.2in, width=1.85in]{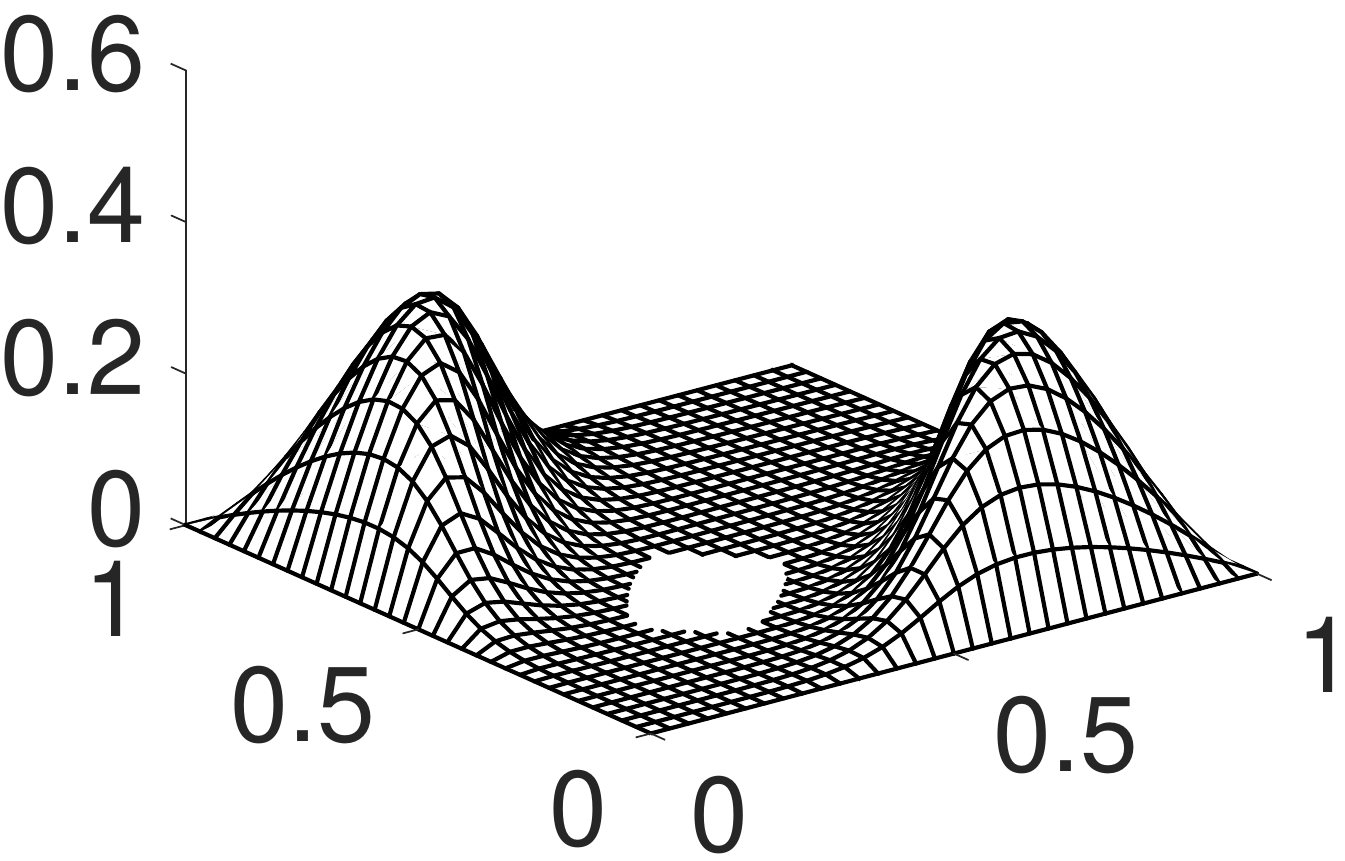}}
    \caption{$\psi_{11}$ approximation}
  \end{subfigure}\\[2ex]
\vspace*{-0.1in}
  \begin{subfigure}{5.6cm}
    \makebox[0pt][l]{\includegraphics[height=1.2in, width=1.85in]{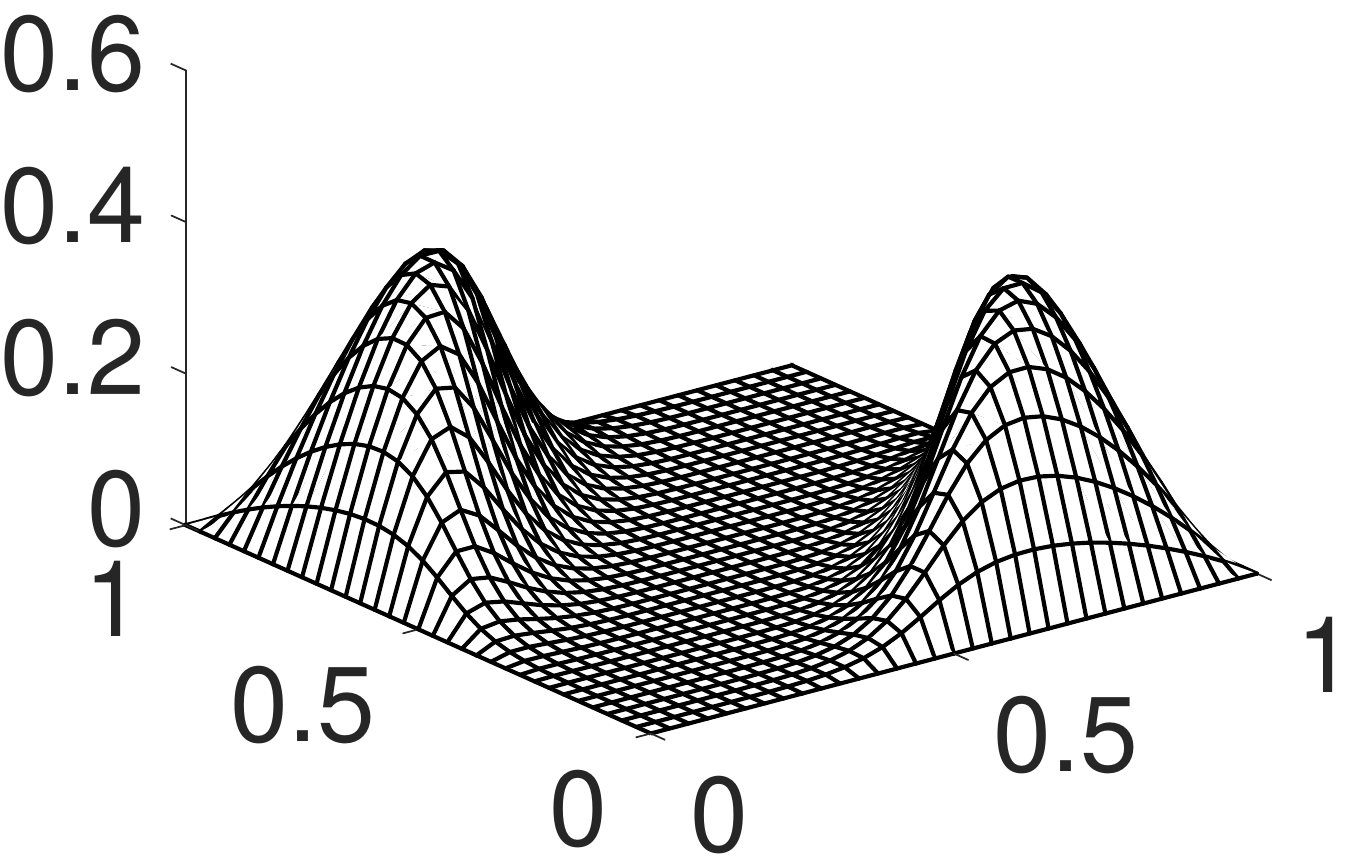}}
    \caption{$\psi_{23}$ approximation}
  \end{subfigure}
  \begin{subfigure}{5.5cm}
       \hspace*{0.08in}\makebox[0pt][l]{\includegraphics[height=1.2in, width=1.85in]{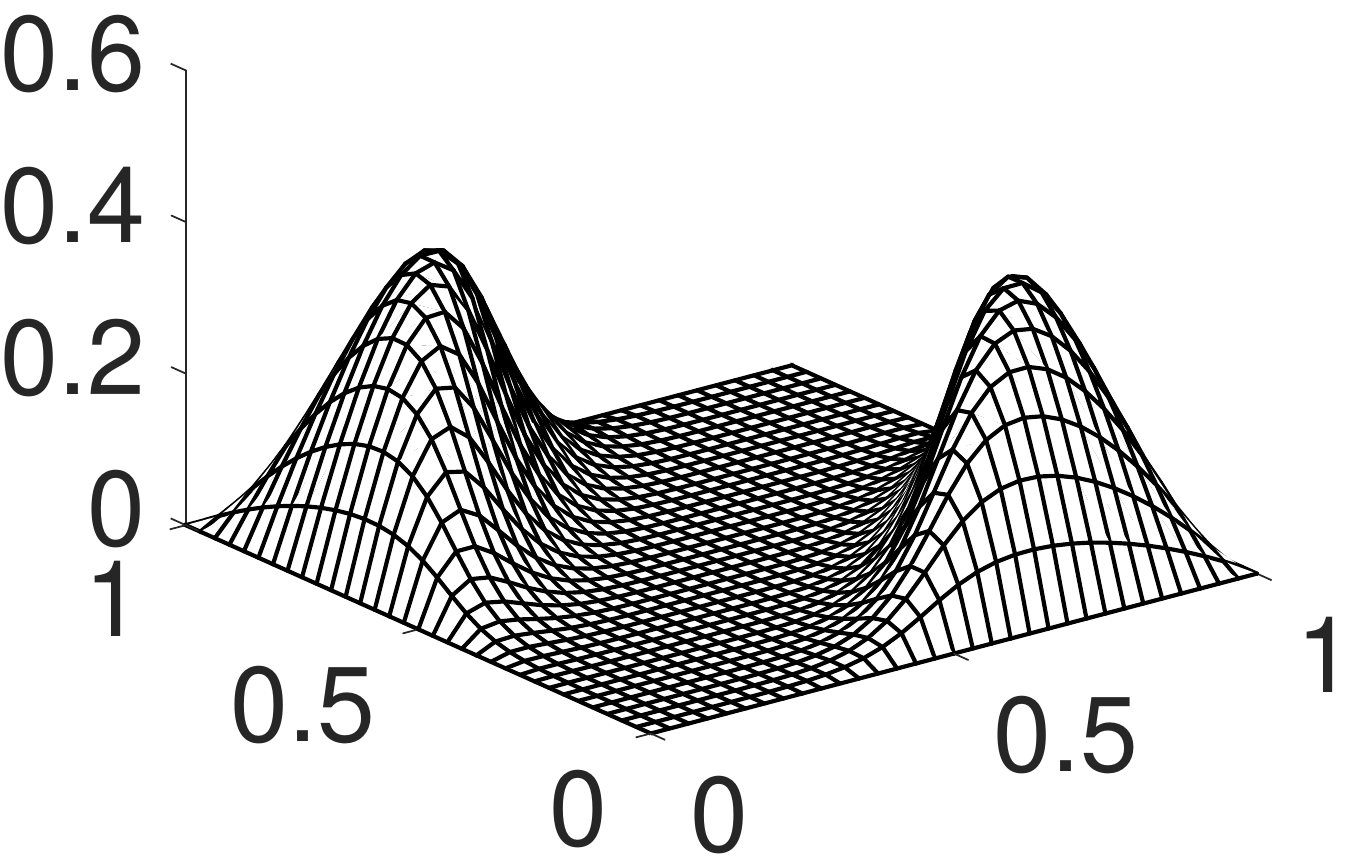}}
    \caption{Analytical solution}
  \end{subfigure}
 \caption {\small Approximation of various modes compared with the analytical solution.}
 \label{fig:fig2}
\end{figure}
\vspace*{-0.2in}
\begin{table}[h!]
\setlength\tabcolsep{5pt} 
\begin{tabular*}{\textwidth}{@{\extracolsep{\fill}} l*{5}{d{2.9}} }
\toprule
$N=2$ &  & \multicolumn{2}{c}{Number of elements}  & & \\
\midrule
 & \multicolumn{1}{c}{$~~~~2^{3}\times2^{3}$} & \multicolumn{1}{c}{$~~~~2^{4}\times2^{4}~$} & \multicolumn{1}{c}{$~~~~2^{5}\times2^{5} ~$} & \multicolumn{1}{c}{$~~~~2^{6}\times2^{6} ~$} & \multicolumn{1}{c}{$~~~~2^{7}\times2^{7}~$} \\
\midrule
$M = 1$ & 1.77952(-01) & 1.75391(-01) & 1.72629(-01) & 1.72591(-01) & 1.72622(-01) \\[1ex]          
$M = 3$ & 8.85117(-02) & 3.72049(-02) & 1.08392(-02) & 5.66948(-03) & 5.07559(-03) \\[1ex]
$M = 5$ & 8.90338(-02) & 3.66827(-02) & 9.31148(-03) & 2.35028(-03) & 6.04344(-04)\\[1ex]
$M = 7$ & 8.90409(-02) & 3.66799(-02) & 9.30586(-03) & 2.34311(-03) & 5.93845(-04)\\
\midrule
Picard & 8.90290(-02) & 3.66693(-02) & 9.29572(-03) & 2.33310(-03) & 5.83828(-04) \\
\midrule
$N=4$ &  & &   & \\
\midrule
$M = 1$ & 1.77952(-01) & 1.75392(-01) & 1.72632(-01) & 1.72594(-01) & 1.72624(-01) \\[1ex]          
$M = 3$ & 8.84999(-02) & 3.71949(-02) & 1.08316(-02) & 5.66720(-03) & 5.07644(-03) \\[1ex]
$M = 5$ & 8.90218(-02) & 3.66722(-02) & 9.30139(-03) & 2.34038(-03) & 5.94682(-04)\\[1ex]
$M = 7$ & 8.90289(-02) & 3.66693(-02) & 9.29571(-03) & 2.33314(-03) & 5.83835(-04)\\
\midrule
Picard & 8.90290(-02) & 3.66693(-02) & 9.29572(-03) & 2.33310(-03) & 5.83828(-04) \\
\bottomrule
\end{tabular*}
\caption{\small{Relative error of series approximation for various $M,N$ and number of elements.}} 
\label{tab:tab2}
\end{table}

\newpage
\subsection{Test 3}:  We set $a(x_1,x_2) = 1+x_1^2 + x_2^2,~r(\psi) = \exp(-\psi^{2}) + \tan^{-1}\psi$ and seek the closed form solution $\psi(x_1,x_2) = \sin(3\pi x_1) \sin(2\pi x_2)$. Moreover, $f$ is chosen accordingly with $r(\psi) = \psi$.
\begin{figure}[h!]
\centering
  \begin{subfigure}{5.5cm}
   \hspace*{0.02in}\makebox[0pt][l]{\includegraphics[height=1.2in, width=1.85in]{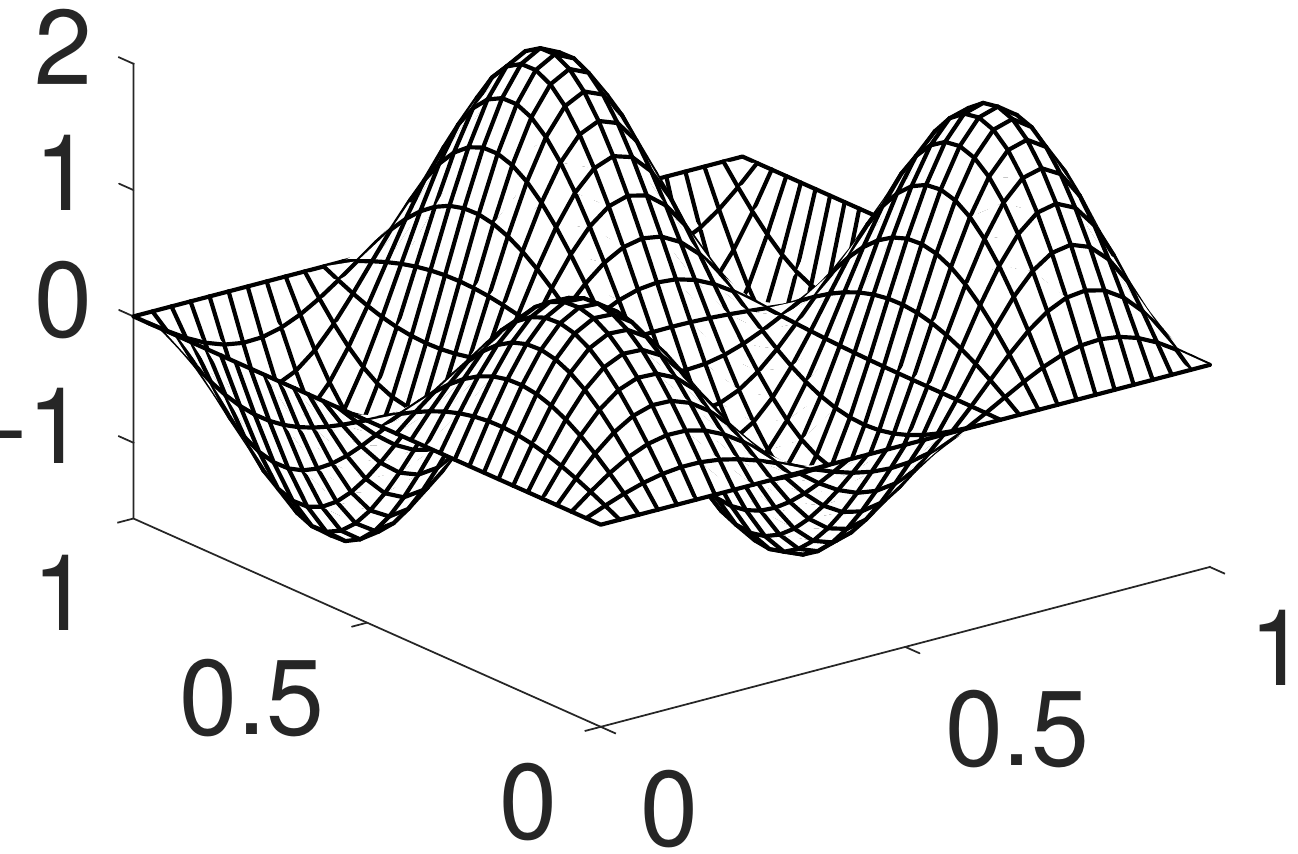}}
    \caption{$\psi_{00}$ approximation}
  \end{subfigure}
  \begin{subfigure}{5.2cm}
    \hspace*{0.06in}\makebox[0pt][l]{\includegraphics[height=1.2in, width=1.85in]{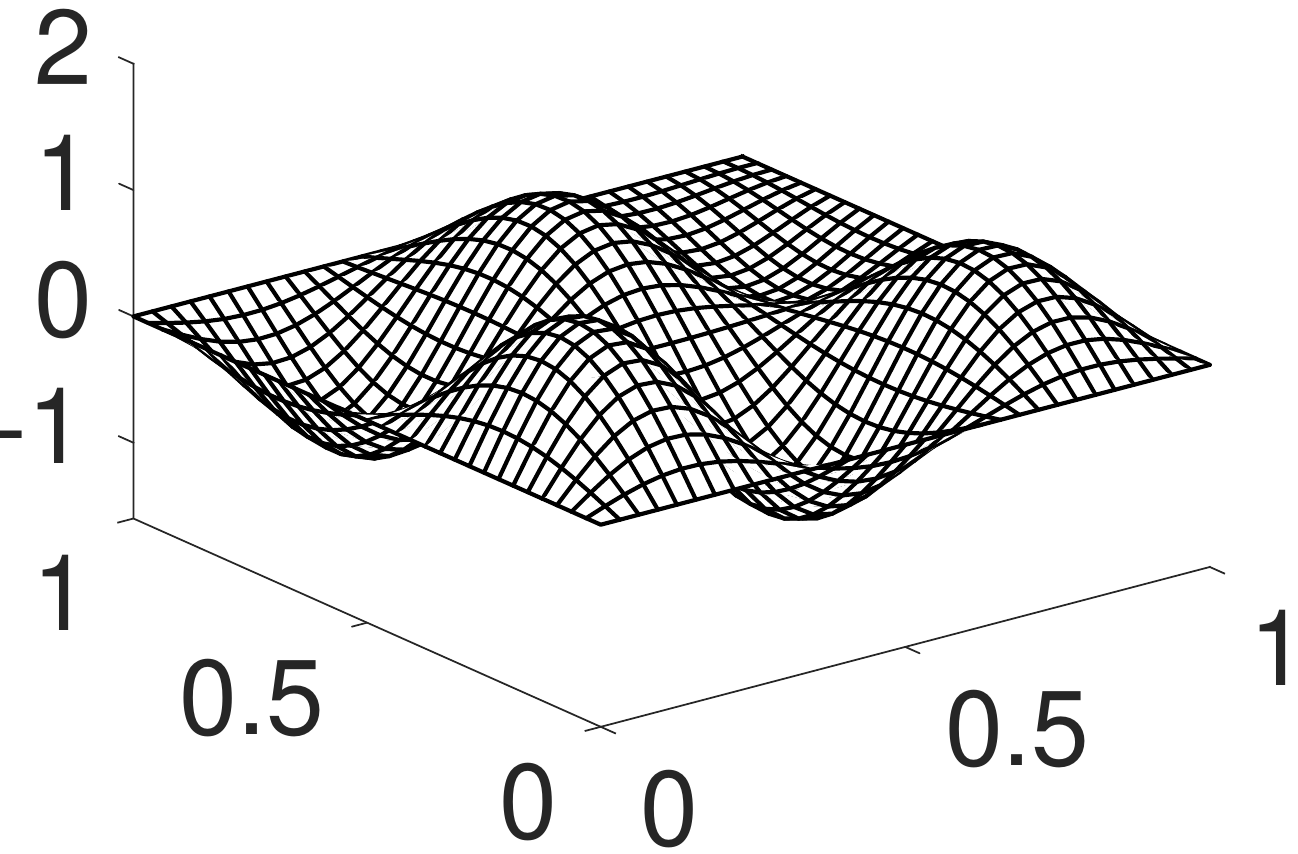}}
    \caption{$\psi_{11}$ approximation}
  \end{subfigure}\\[2ex]
\vspace*{-0.1in}
  \begin{subfigure}{5.6cm}
    \makebox[0pt][l]{\includegraphics[height=1.2in, width=1.85in]{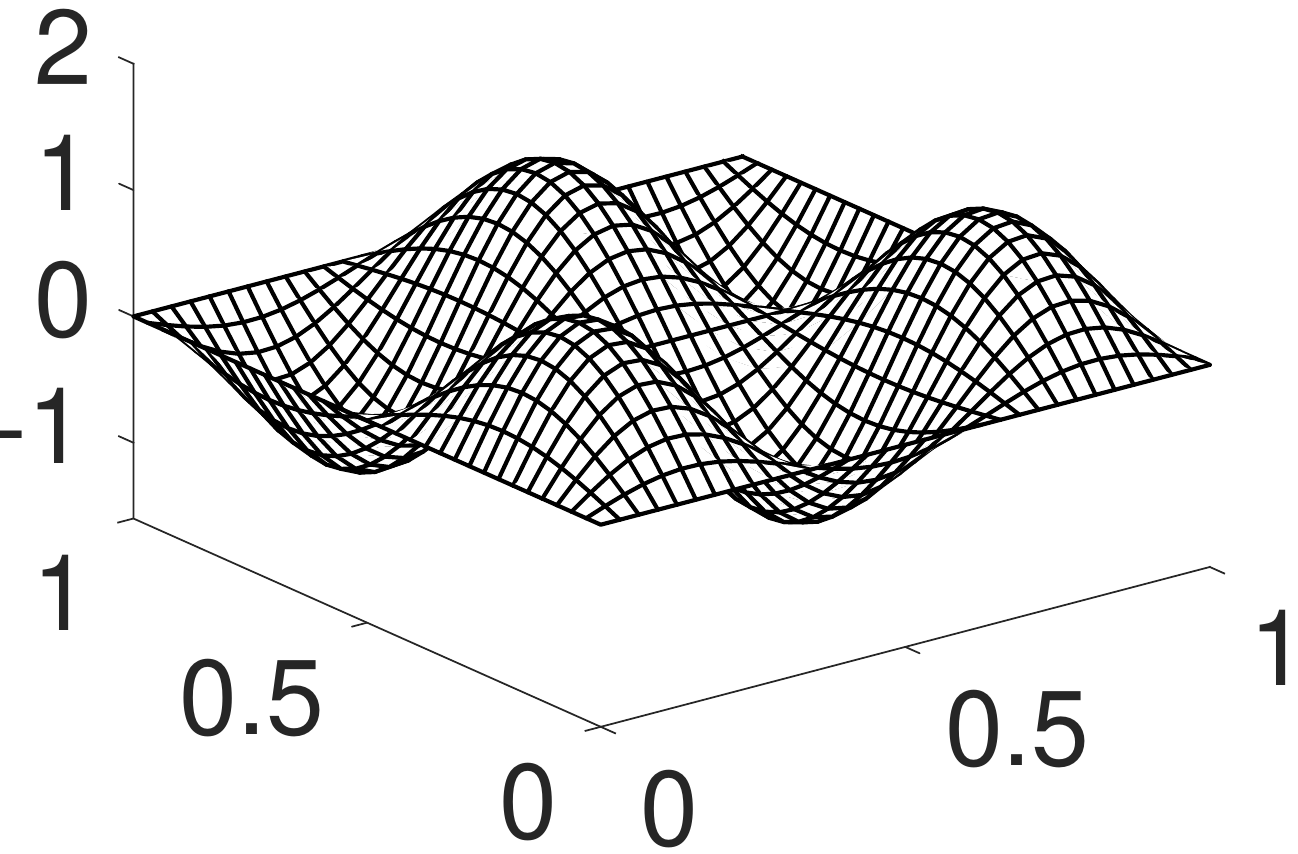}}
    \caption{$\psi_{24}$ approximation}
  \end{subfigure}
  \begin{subfigure}{5.5cm}
       \hspace*{0.08in}\makebox[0pt][l]{\includegraphics[height=1.2in, width=1.85in]{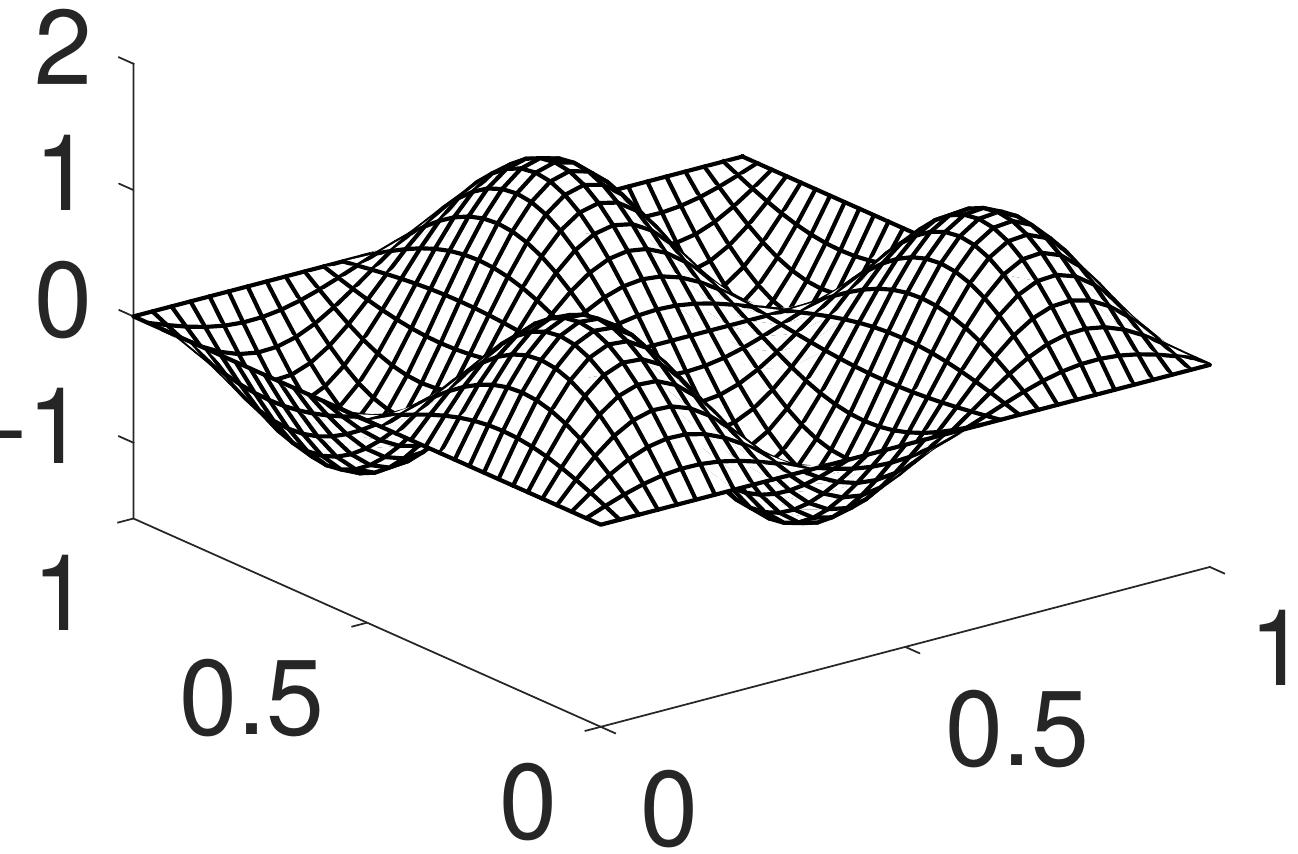}}
    \caption{Analytical solution}
  \end{subfigure}
 \caption {\small Approximation of various modes compared with the analytical solution.}
 \label{fig:fig2}
\end{figure}
\begin{table}[h!]
\setlength\tabcolsep{5pt} 
\begin{tabular*}{\textwidth}{@{\extracolsep{\fill}} l*{5}{d{2.9}} }
\toprule
$N=2$ &  & \multicolumn{2}{c}{Number of elements} &  & \\
\midrule
 & \multicolumn{1}{c}{$~~~~2^{3}\times2^{3}$} & \multicolumn{1}{c}{$~~~~2^{4}\times2^{4}~$} & \multicolumn{1}{c}{$~~~~2^{5}\times2^{5} ~$} & \multicolumn{1}{c}{$~~~~2^{6}\times2^{6} ~$} & \multicolumn{1}{c}{$~~~~2^{7}\times2^{7}~$} \\
\midrule
$M = 1$ & 3.67061(-01) & 3.25203(-01) & 3.18984(-01) & 3.17578(-01) & 3.17235(-01) \\[1ex]          
$M = 3$ & 1.02834(-02) & 3.62913(-02) & 2.20172(-02) & 1.93606(-02) & 1.88047(-02) \\[1ex]
$M = 5$ & 9.31325(-02) & 2.41869(-02) & 6.25023(-03) & 1.74825(-03) & 6.98383(-04)\\[1ex]
$M = 7$ & 9.29634(-02) & 2.40161(-02) & 6.06422(-03) & 1.52422(-03) & 3.86019(-04)\\
\midrule
Picard & 9.29576(-02) & 2.40101(-02) & 6.05810(-03) & 1.51804(-03) & 3.79729(-04) \\
\midrule
$N=4$ &  & &   & \\
\midrule
$M = 1$ & 3.67060(-01) & 3.25202(-01) & 3.18984(-01) & 3.17578(-01) & 3.17235(-01)\\[1ex]          
$M = 3$ & 1.02830(-02) & 3.62881(-02) & 2.20151(-02) & 1.93592(-02) & 1.88035(-02)\\[1ex]
$M = 5$ & 9.31286(-02) & 2.41828(-02) & 6.24614(-03) & 1.74434(-03) & 6.95296(-04)\\[1ex]
$M = 7$ & 9.29585(-02) & 2.40103(-02) & 6.06010(-03) & 1.52007(-03) & 3.80026(-04)\\
\midrule
Picard & 9.29576(-02) & 2.40101(-02) & 6.05810(-03) & 1.51804(-03) & 3.79729(-04) \\
\bottomrule
\end{tabular*}
\caption{\small{Relative error of series approximation for various $M,N$ and number of elements.}} 
\label{tab:tab2}
\end{table}

\newpage

\section{Discussions and conclusion}
The numerical results show an impressive agreement between the series approximation and the exact solutions. Since the lower order modes do not inherit sufficient spatial variability from $a$, their approximations are distinguishable from the analytical  solutions. Once the variability is well resolved, it takes only a few terms to obtain an accurate approximation. Moreover, the contribution of larger values of $M$ and $N$ to the overall error is insignificant as the finite element error dominates. The errors recorded show that using an excessive number of terms to construct the partial solution is obviously an overkill because the gain is insignificant compared to the computational cost. Overall, the method provides accurate approximations and the variational context provides a flexible framework to study its convergence. For the benchmark problem, the method allows the inversion of nonlinear differential operators to be achieved through a series of inversions of the discrete Laplacian. This is appealing because the nontrivial task of calculating the associated Jacobian matrix is avoided entirely. For linear reactions, the method is shown to coincide with the Picard iteration method. Even though the equivalence has not been established for a general nonlinear reaction, the numerical results of the two approaches are largely consistent.

\bibliographystyle{siam}
\bibliography{biblio} 
\end{document}